\documentclass[a4paper,11pt]{article}

\usepackage[top=3.0cm,bottom=3.0cm,left=2.5cm,right=2.5cm]{geometry}
\usepackage{tikz}
\usepackage{graphics,epsfig,psfrag}
\usepackage{graphicx}
\usepackage{amsfonts}
\usepackage{amssymb}
\usepackage{amsmath}
\usepackage{amsthm}
\usepackage[english]{babel}
\usepackage{ctex}
\usepackage{hyperref}
\usepackage{cases}
\usepackage{authblk}
\usepackage{mathtools}
\usepackage{pifont}

\newtheorem{lemma}{Lemma}[section]
\newtheorem{theorem}[lemma]{Theorem}

\newtheorem{claim}[]{\noindent Claim}[section]

\newtheorem{conjecture}[lemma]{Conjecture}

\begin{document}

\title{A simple proof of the existence of complete bipartite graph immersion in graphs with independence number two}

        \author[1]{Rong Chen\footnote{Email: rongchen@fzu.edu.cn (R. Chen).}}
		\author[1]{Zijian Deng\footnote{Email: zj1329205716@163.com.(corresponding author).}}

\affil[1]{Center for Discrete Mathematics, Fuzhou University,  China}

	\date{}
	\maketitle
	\begin{abstract}
 Hadwiger's conjecture for the immersion relation posits that every graph $G$ contains an immersion of the complete graph $K_{\chi(G)}$. Vergara showed that this is equivalent to saying that every $n$-vertex graph $G$ with $\alpha(G)=2$ contains an immersion of the complete graph on $\lceil\frac{n}{2}\rceil$ vertices. Recently, Botler et al. showed that
every $n$-vertex graph $G$ with $\alpha(G)=2$ contains every complete bipartite graph on
$\lceil\frac{n}{2}\rceil$ vertices as an immersion. In this paper, we give a much simpler proof of this result.

	\end{abstract}
	
	

\footnote{This research was partially supported by grants from the National Natural Sciences
Foundation of China (No. 11971111).}

\section{\bf Introduction and Preliminaries}
All graphs considered in this paper are finite and simple. Let $G=(V,E)$ be a graph. An \emph{independent set} (or a \emph{stable set}) is a subset of $V(G)$ that are pairwise non-adjacent. Let $\alpha(G)$ be the \emph{independence number} of $G$, that is the size of the largest independent set. Let $K_{n}$ be a clique on $n$ vertices and $\chi(G)$ the \emph{chromatic number} of $G$. 

For any set $X\subseteq V(G)$, let $G[X]$ denote the subgraph of $G$ induced by $X$.
Given disjoint vertex sets $A$ and $B$, we say that $A$ is \emph{complete} to $B$ if  each vertex in $A$ is adjacent to all vertices in $B$. Analogously, $A$ is \emph{anti-complete} to $B$ if there exists no edge between $A$ and $B$.
We use $N_{G}(a)$ to denote the set of neighbors of vertex $a$ in graph $G$ and set $N_{G}[a]:=N_{G}(a)\cup \{a\}$. When there is no danger of confusion, the subscript $G$ is omitted.

If a graph isomorphic to $H$ can be obtained from a graph $G$ by deleting vertices or edges and contracting edges, then $H$ is a \emph{minor} of $G$, written by $G\succeq_m H$.
Hadwiger's Conjecture \cite{HU} says that every graph $G$ without loops contains the complete graph $K_{\chi(G)}$ as a minor. This conjecture has been proven to be true for graphs with a chromatic number of $\chi(G)\leq 6$, but is likely difficult to prove for larger values, given that proofs for the cases of $\chi(G)=5$ and $\chi(G)=6$ already depend on the Four Color Theorem.


Hadwiger's conjecture remains unsolved for a specific class of graphs, namely those with independence number of at most two. The significance of this class was first pointed out by Duchet and Meyniel \cite{PH}  and Mader independently, and later highlighted in Seymour's recent survey \cite{P}.
Plummer, Stiebitz, and Toft \cite{MM} showed that, for this class of graphs, the conjecture can be restated as follows: every $n$-vertex graph $G$ with  $\alpha(G) \leq 2$ contains a minor isomorphic to $K_{\lceil \frac{n}{2} \rceil}$. 
Woodall \cite{W} in 2001 (and independently Seymour \cite{AN}) proposed the following conjecture. 

\begin{conjecture}{\rm(\cite{W,AN})}\label{conj4}
For any positive integer $\ell$ with $\ell< \chi(G)$, a graph $G$ contains a $K_{\ell, \chi(G)-\ell}$-minor.
\end{conjecture}


\noindent Chen and Deng \cite{RZ} recently showed that Conjecture \ref{conj4} holds for graphs with independence number two. 

This article explores an alternative ordering of graphs called the \emph{immersion order}, which is defined through the use of \emph{edge lifts}. When two edges $uv$ and $vw$ are adjacent in a graph $G$ but $u\neq w$ and $uw\notin E(G)$, a lift can be performed by deleting $uv$ and $vw$ and adding the edge $uw$. A graph $H$ is said to be {\em immersed} in $G$, denoted by $G\succeq_iH$, if $H$ can be obtained from $G$ through a series of lifts of edges and deletions of vertices and/or edges. Another way of stating this is that $G$ contains $H$ as an immersion.

This definition is equivalent \cite{FM} to the existence of an injective function $\phi: V(H)\rightarrow V(G)$ that satisfies the following two conditions:

(1) For any $uv\in E(H)$, there is a path in $G$ denoted as $P_{uv}$ that connects the vertices $\phi(u)$ and $\phi(v)$.

(2) The paths $\{P_{uv} : uv\in E(H)\}$ are pairwise edge-disjoint.

\noindent We call the vertices in $\phi(V(H))$ \emph{branch vertices} of the immersion.

Notice that if we modify the second condition of the definition of immersion to require that the paths $P_{uv}$ be internally vertex disjoint rather than just edge disjoint, then $G$ is said to contain $H$ as a topological minor. It follows that if $G$ contains $H$ as a topological minor, then $G$ also contains $H$ as an immersion. It should be emphasized that the minor order and the immersion order are not comparable.

In recent years, the immersion order has become a subject of significant interest in the field of graph theory. Much of this attention has been focused on the conjecture proposed by Lescure and Meyniel \cite{FH} (see also \cite{FM}), which is an immersion analogue of Hadwiger's conjecture.

\begin{conjecture}\label{conj1}
(\cite{FH}) Every graph $G$ contains $K_{\chi(G)}$ as an immersion.
\end{conjecture}

When the chromatic number $\chi(G) \leq 4$, Conjecture \ref{conj1} holds, due in part to the truth of Haj\'{o}s' Conjecture for these cases \cite{GA}. For larger chromatic numbers, the conjecture has been established for cases $5 \leq \chi(G) \leq 7$, by both Lescure and Meyniel \cite{FH} and DeVos, Kawarabayashi, Mohar, and Okamura \cite{MK}.
However, beyond these cases, there are few non-trivial instances where the conjecture is known to hold. For example, it has been proven that the conjecture is true for the line graphs of simple graphs, but it remains unsolved for the line graphs of general multigraphs \cite{MJ}. 
When we restrict Conjecture \ref{conj1} to graphs with independence number of at most two, Vergara proposed the following conjecture.



\begin{conjecture}\label{conj3}(\cite{V})
Every graph $G$ on $n$ vertices with $\alpha(G)\leq2$ contains $K_{\lceil \frac{n}{2}\rceil}$ as an immersion.
\end{conjecture}

Vergara \cite{V} showed that if Conjecture \ref{conj1} is restricted to graphs with independence number of at most two, then Conjecture \ref{conj1} and Conjecture \ref{conj3} become equivalent.
In the same paper, she showed that every $n$-vertex graph $G$ with $\alpha(G)\leq2$ contains $K_{\lceil \frac{n}{3}\rceil}$ as an immersion, which provides strong evidence for the conjecture.
Lately, Gauthier, Le, and Wollan \cite{GT} were able to improve Vergara's result to $K_{2{\lfloor \frac{n}{5}\rfloor}}$. 
Recently, Botler et al. showed the following result true. 

\begin{theorem} \label{thm4}(\cite{FA}) Let $G$ be an $n$-vertex graph with $\alpha(G)=2$. For any positive integer $\ell$ with $2\ell \leq \lceil\frac{n}{2}\rceil$, we have $G \succeq_{i} K_{\ell,\lceil\frac{n}{2}\rceil- \ell}$.
\end{theorem}
By a result in  \cite{FA}, Theorem \ref{thm4} implies that each graph $G$ with $\alpha(G)=2$ contains $K_{\ell,\chi(G)- \ell}$ as an immersion for any positive integer $\ell$ with $\ell< \chi(G)$. That is, the immersion version of Conjecture \ref{conj4} holds  for graphs with independence number two.
In this paper, we give a much simpler proof of Theorem \ref{thm4}. Our strategy is to first show that each counterexample to Theorem \ref{thm4} does not contain an induced $4$-vertex-cycle, 
and then use Quiroz's result, namely Theorem \ref{lem1}, to take a shortcut.

\begin{theorem} \label{lem1}(\cite{D}) Let $G$ be a $n$-vertex graph with $\alpha(G)\leq 2$.
If $G$  does not contain an induced $4$-vertex-cycle, 
then $G$  contains $K_{\lceil \frac{n}{2}\rceil}$ as an immersion.
\end{theorem}

In fact, Chen and Deng in \cite{RZ} showed that each graph $G$ with $\alpha(G)\leq2$ contains a $K^{\ell}_{\ell,\chi(G)- \ell}$-minor for each positive integer $\ell$ with $2\ell \leq \chi(G)$, where  $K_{m,n}^m$ is obtained from the disjoint union of a $K_m$ and an independent set on $n$ vertices by adding all of the possible edges between them. Considered of this fact, we conjecture that a similar result also holds for immersion. This conjecture is trivial for $\ell=1$ as $\Delta(G)\geq \chi(G)-1$ implies that $G$ contains a $K^{1}_{1,\chi(G)-1}$-immersion, where $\Delta(G)$ is the maximum degree of $G$.
As more evidence for this conjecture, Botler et al. \cite{FA} showed that this conjecture holds for $\ell = 2$.








\smallskip
\section{Proof of Theorem \ref{thm4}}


Assume that Theorem \ref{thm4} is not true. Let $G$ be a counterexample to Theorem \ref{thm4} with $|V(G)|$ as small as possible. 


\begin{claim}\label{claim0}
$n\geq4\ell-1$ is odd.
\end{claim}

\begin{proof}
Assume that $n$ is even. Since Theorem \ref{thm4} holds for $G - \{v\}$ for any $v\in V(G)$, Theorem \ref{thm4} holds for $G$ as $\lceil\frac{n}{2}\rceil=\lceil\frac{n-1}{2}\rceil$, a contradiction. Therefore, $n$ is odd. Since $\ell \leq \lceil\frac{n}{2}\rceil- \ell=\frac{n+1}{2}-\ell$, we have $n\geq4\ell-1$.
\end{proof}

Set $n:=2m+1$. Then $\lceil\frac{n}{2}\rceil=m+1$ and $\ell\leq \lfloor\frac{m+1}{2}\rfloor$.
\begin{claim}\label{claim1}
For any $uv\notin E(G)$, $|N(u)\cap N(v)|\leq \ell-2$.
\end{claim}
\begin{proof}
Suppose to the contrary that $|N(u)\cap N(v)|\geq \ell-1$ for some non-adjacent vertices $\{u,v\}$.
By the minimality of $G$, the graph $G-\{u,v\}$ has a $K_{\ell,m-\ell}$-immersion $K'$. Let $L$ be the branch vertices of the part of $K'$ with size $\ell$.
When $u$ or $v$ is complete to $L$, the graph $G$ has a $K_{\ell,m+1-\ell}$-immersion, a contradiction.
So neither $u$ nor $v$ is complete to $L$. Since $\alpha(G)=2$ and $uv\notin E(G)$, $L$ can be partitioned into $A,B,C$ such that $u$ is complete to $A$ but $v$ is anti-complete to $A$, $\{u,v\}$ is complete to $B$, and $v$ is complete to $C$ but $u$ is anti-complete to $C$.
Let $D$ be the subset of $V(G)-\{u,v\}-L$ consisting of vertices complete to $\{u,v\}$.
Since $|L|=|A\cup B\cup C|=\ell$ and $A\neq \emptyset$, we have $|B\cup C|\leq \ell-1$. Moreover, since $|B\cup D|=|N(u)\cap N(v)|\geq \ell-1$, we have $|D|\geq |C|$.
Let $f$ be an injection from $C$ to $D$. Then $\{u$-$f(c)$-$v$-$c: c\in C\}$ are edge-disjoint paths from $u$ to each vertex of $C$. So $G$ has a $K_{\ell,m+1-\ell}$-immersion as $u$ is complete to $A\cup B$, which is a contradiction.
\end{proof}

We say a graph $H$ is {\em $\alpha$-critical} if for each edge $e$ of $H$ we have $\alpha(H\backslash e)>\alpha(H)$. Without loss of generality we may further assume that $G$ is $\alpha$-critical. Hence, for each edge $uv\in E(G)$, there exists some $z\in V(G)-\{u,v\}$ such that $z$ is anti-complete to $\{u,v\}$.

For any two non-adjacent vertices $x$ and $y$, set \[C := N(x)\cap N(y),\ X:=V(G)-N[y],\ Y:= V(G)-N[x].\] Since $\alpha(G)=2$, we have $x\in X$, $y\in Y$ and $G[X]$ and $G[Y]$ are cliques.  By Claim \ref{claim1}, we have $|C|\leq \ell-2$.
Let $X'_C$ be the subset of $X$ consisting of vertices complete to $C$, and set $X''_C: = X \backslash X'_C$. Let $Y'_C$ and $Y''_C $ be defined similarly. Note that $x\in X'_C$ and $y\in Y'_C$.
Similarly, for any $a\in C$, let $X'_a$ denote the subset of $X$ consisting of vertices complete to $a$, and set $X''_a: = X \backslash X'_a$. Let $Y'_a$ and $Y''_a $ be defined similarly. Evidently,  $X'_C \subseteq X'_a$ and $Y'_C \subseteq Y'_a$. Moreover, since $\alpha(G)=2$, the set $X''_a$ is complete to $Y''_a$.

\begin{claim}\label{claim2}
For every vertex $a \in C$, we have $X''_a\neq \emptyset$ and $Y''_a \neq \emptyset.$
\end{claim}

\begin{proof}
Suppose not. By symmetry we may assume that $a$ is complete to $Y$. Since $x$ is complete to $C \cup (X\backslash \{x\})$ and $ax\in E(G)$, there exists no vertex anti-complete to $\{a,x\}$, which contradicts to the fact that $G$ is $\alpha$-critical. 
\end{proof}

\begin{claim}\label{claim3}
For any vertex $a \in C$, we have that $|X'_C| \leq |X'_a| \leq \ell- 2$ and $|Y'_C| \leq |Y'_a| \leq \ell-2$. Furthermore, we have $|X''_a | \geq m+4-|Y|$ and $|Y''_a | \geq m+4-|X|$.
\end{claim}

\begin{proof}
By symmetry it suffices to show that Claim \ref{claim3} holds for $X'_a$ and $X''_a $. 
Note that $|X'_C|\leq|X'_a|$ as $X'_C\subseteq X'_a$.
By Claim \ref{claim2}, there exists a vertex $a_1 \in X''_a$.
For the non-adjacent pair $\{a,a_1\}$, since $G[X]$ is a clique, $X'_a\subseteq N(a)\cap N(a_1)$. Thus $|X'_a|\leq \ell-2$ by Claim \ref{claim1}.
Moreover, since $n\geq4\ell-1$ and $|C|\leq \ell-2$, we have $$|X''_a |=n-|C|-|X'_a|-|Y|\geq n-2(\ell-2)-|Y|\geq m+4-|Y|.$$ 
\end{proof}

\begin{claim}\label{claim4}
$G[C]$ is a clique.
\end{claim}

\begin{proof}
Suppose to the contrary that $G[C]$ is not a clique. Let $u,v\in C$ with $uv\notin E(G)$.
Since $\alpha(G)=2$, we have $X''_v \cap X''_u$, $Y''_v \cap Y''_u$$=\emptyset$.
By symmetry we may suppose that $|X''_v|\geq |Y''_v|$.
Since $|C|, |X'_u|, |Y'_v|\leq \ell-2$ by Claims \ref{claim1} and \ref{claim3},
$|X''_u\cup Y''_v|\geq n-3(\ell-2)\geq 4\ell-1-(3\ell-6)>\ell$.
Since $X''_v \cap X''_u = \emptyset$ and $|X''_v|\geq |Y''_v|$, $|X''_u\cup X''_v|>\ell$.
Choose $X_1\subseteq X''_v$ such that $|X''_u\cup X_1|=\ell$. 
Choose $Y_1\subseteq Y''_u$ such that $|Y_1|=m+1-|X|$.
Since $|Y''_u|> m+4-|X|$ by Claim \ref{claim3}, such $Y_1$ exists.
Our aim is to find a $K_{\ell,m+1-\ell}$-immersion with $X''_u\cup X_1$ and $(X\backslash (X''_u\cup X_1)) \cup Y_1$ as its part of branch vertices of size $\ell$ and $m+1-\ell$, respectively.

Since $G[X]$ is a clique, $X''_u\cup X_1$ is complete to $X\backslash (X''_u\cup X_1)$. Moreover, since $\alpha(G)=2$, the set $X''_v$ is complete to $Y''_v$, and $X''_u$ is complete to $Y''_u$. Hence, it suffices to show that there are edge-disjoint paths joining each vertex in $X_1$ to each vertex in $Y_1$. Set $X_1:=\{x_1,...,x_{|X_1|}\}$, $Y_1:=\{y_1,...,y_{|Y_1|}\}$ and $Y''_v:=\{z_1,...,z_{|Y''_v|}\}$.
Since $Y_1\subseteq Y''_u$, we have $Y_1 \cap Y''_v= \emptyset$.
For each pair of vertices $\{x_i, y_j\}$, since $G[Y]$ is a clique, $P_{ij}=x_i$-$z_{i+j}$-$y_j$ is a path, where the subscript of $z_{i+j}$ is taken modulo $|Y''_v|$.
Since $|X''_u\cup Y''_v|>\ell=|X''_u\cup X_1|$, we have $|Y''_v|> |X_1|$.
Moreover, by Claim \ref{claim3}, $|Y''_v|> m+1-|X|=|Y_1|$. Hence, $\{P_{ij}:\  1\leq i\leq |X_1|,\ 1\leq j\leq |Y_1|\}$ are edge-disjoint paths. Hence, $G$ has a $K_{\ell,m+1-\ell}$-immersion, which is a contradiction.
\end{proof}

Since $xy\notin E(G)$ by the choice of $x,y$, there is no induced 4-vertex-cycle in $G$ containing $x$, $y$ by Claim \ref{claim4} and the definition of $C$. Moreover, by the arbitrary choice of $x,y$ and symmetry, the graph $G$ does not contain an induced 4-vertex-cycle. 
Thus, $G\succeq_i K_{\lceil\frac{n}{2}\rceil}$ by Theorem \ref{lem1},  which is a contradiction.

\section{Acknowledgments}
We would like to express our gratitude to two anonymous reviewers for their diligent review and valuable suggestions, which significantly enhanced the clarity and presentation of this paper.


	
	
	\vspace{5pt}

\end{document}